\documentclass[a4paper,11pt]{amsart}%{article}

\usepackage{amsmath,amsthm,array,xypic}
\usepackage{amssymb}
\usepackage{mathtools}
\usepackage{enumerate}
\usepackage[mathcal]{euscript}
\usepackage[margin=3cm]{geometry}
 
\usepackage[utf8]{inputenc}
\usepackage{tgtermes}
\usepackage{mathptmx}
\usepackage[T1]{fontenc}
\usepackage{stackrel}
\usepackage{pict2e,picture}

\newtheorem{Theorem}{Theorem}[section]
\newtheorem{Lemma}[Theorem]{Lemma}
\newtheorem{Proposition}[Theorem]{Proposition}

\newtheorem{Definition}[Theorem]{Definition}

\theoremstyle{definition}
\newtheorem{Example}[Theorem]{Example}
 
\newtheorem{Remark}[Theorem]{Remark}

\newcommand{\CC}{\mathbb{C}}

\newcommand{\ZZ}{\mathbb{Z}}
\newcommand{\NN}{\mathbb{N}}

\newcommand{\m}{\mathfrak{m}}

\newcommand{\Spec}{\operatorname{Spec}}
\newcommand{\Specm}{\operatorname{Specm}}
\newcommand{\ad}{\operatorname{ad}}
\newcommand{\kar}{\operatorname{char}}
\newcommand{\Span}{\operatorname{Span}}

\newcommand{\Gr}{{\operatorname{Gr}}}

\newcommand{\leftbimod}[3]{\vphantom{#1}^{#2}{\kern-#3pt #1}}

\numberwithin{equation}{subsection}

\title{Two properties of endomorphisms of Weyl algebras}

\author{Niels Lauritzen and Jesper Funch Thomsen} 
\begin{document}
\maketitle 

\begin{abstract}
We show that endomorphisms of Weyl algebras over fields
of characteristic zero are flat and that birational endomorphisms
are automorphisms by reducing to positive characteristic. We
also give examples showing that endomorphisms of Weyl algebras are not
in general flat over fields of positive characteristic.
\end{abstract}

\section*{Introduction}

We prove that endomorphisms of Weyl algebras over a field of
characteristic zero are flat.  More precisely, let $A$ be the $n$-th
Weyl algebra over a field of characteristic zero and
$\varphi:A\rightarrow A$ an endomorphism with $S = \varphi(A)$.  We
prove that $S \subset A$ is a flat ring extension in the sense of
\cite[Ch. 2, \S2.8]{Bjork} i.e., $A$ is flat both as a left and a
right $S$-module. The endomorphism $\varphi$ gives an extension of
division rings $Q(S) \subset Q(A)$. We call $\varphi$ birational if
$Q(S) = Q(A)$ and show that birational endomorphisms of Weyl algebras
over fields of characteristic zero are automorphisms. This is a
non-commutative analogue (over fields of characteristic zero) of
Keller's classical result that birational endomorphisms of affine
spaces with invertible Jacobian are automorphisms \cite{Keller}. In
general, $Q(S)\subset Q(A)$ is an extension whose left- and right
dimensions are bounded by $\deg(\varphi)^{2n}$ (see Proposition
\ref{PropDegBound}).

The key component in our approach is reduction to positive
characteristic, where the $n$-th Weyl algebra is finite free over its
center, which is a polynomial ring in $2n$ variables \cite{Revoy} with canonical
Poisson bracket coming from the commutator in
the lifted Weyl algebra (see \cite{Kontsevich1} and \S \ref{antiquantization} of this paper).

The Dixmier conjecture \cite[\S11.1]{Dixmier} states that an
endomorphism of the first Weyl algebra, i.e. $n = 1$ above, 
over a field of characteristic zero
is in fact an automorphism. 
The Jacobian conjecture states 
that an endomorphism of affine $n$-space with invertible Jacobian over a field of characteristic zero is
an automorphism for $n\geq 2$.
The natural extension of the Dixmier conjecture to
$n\geq 1$ is a non-commutative analogue\footnote{The Dixmier conjecture for the $n$-th Weyl algebra implies the
  Jacobian conjecture for affine $n$-space, see~\cite[p.~297]{Bass}. The Jacobian conjecture for affine $2n$-space implies the Dixmier conjecture for the $n$-th Weyl algebra, see~\cite{Kontsevich1} and \cite{Tsuchimoto1}.} of the Jacobian
conjecture and seems inherently intractable. 

It is known over a field of arbitrary characteristic, that an endomorphism of affine $n$-space
with invertible Jacobian is flat \cite[(2.1) \textsc{Theorem}]{Bass}\cite{Miyanishi}.
We give examples of endomorphisms of the first Weyl algebra over fields of positive characteristic for which
flatness fails.

A proof of flatness of endomorphisms of Weyl algebras over fields of characteristic zero was presented in \cite{Tsuchimoto3}, but it seems to contain 
a mistake (see \S \ref{FlatCharZero} of this
paper for further details), which we at present do not know how to circumvent.

\section{Preliminaries}

 Most of this section is aimed at introducing the Weyl algebra over
 commutative rings of prime characteristic and the reduction from zero
 to positive characteristic. Except for a few results, we have
 deliberately done this in some detail to make the paper self
 contained. The study of the Weyl algebra over rings of prime
 characteristic was initiated in \cite{Revoy}.

Throughout this paper $\NN$ denotes the natural numbers $\{0, 1, 2,
\dots\}$ and $R$ a commutative ring.

\subsection{The Weyl algebra over a commutative ring}

\begin{Proposition}\label{Prop00}
Let $S$ be a ring and $\partial, x\in S$ with $[\partial, x] = \partial x - x \partial = 1$. With the convention $\partial^r = x^r = 0$ for $r<0$,
\begin{enumerate}[(i)]
\item
$$ 
\ad(\partial)^i \ad(x)^j \left( x^m\partial^n\right) = (-1)^j i!\, j!\, \binom{m}{i} \binom{n}{j} x^{m-i} \partial^{n-j},
$$
\item
$$
[\partial^m, x^n]=\sum_{k\geq 1} k! \binom{m}{k} \binom{n}{k} x^{n-k} \partial^{m-k}
$$
\end{enumerate}
for $i, j, m, n\in \NN$.
\end{Proposition}
\begin{proof}
The formula in $(i)$ follows using that $\ad(x)$ and $\ad(\partial)$ are derivations of $S$ with 
$\ad(\partial)(x) = [\partial, x] = -\ad(x)(\partial) = 1$.

The identity in $(ii)$ goes back to \cite[\textsc{Theorem}
  XIII]{Littlewood}. A proof may be given by induction using that
$[\partial^{m+1}, x^n] = [\partial, x^n] \partial^m + \partial
[\partial^m, x^n]$ and $[\partial^m, x^{n+1}] = [\partial^m, x] x^n +
x [ \partial^m, x^n]$ (see \cite[Lemma 2.1]{Dixmier}).
\end{proof}

\begin{Definition}
The $n$-th Weyl algebra $A_n(R)$ over $R$ 
is the
free $R$-algebra on $x_1, \dots, x_n, \partial_1, \dots, \partial_n$ 
with relations
\begin{align}\label{commrules}
\begin{split}
[x_i, x_j] &= 0\\
[\partial_i, \partial_j] &= 0\\
[\partial_i, x_j] &= \delta_{ij}
\end{split}
\end{align}
for $1\leq i, j \leq n$, where $\delta_{ij}$ denotes the Kronecker delta.
By abuse of notation we let $x_i$ and $\partial_i$ denote their canonical images in $A_n(R)$ for $i = 1, \dots, n$. 
\end{Definition}

\begin{Definition}
For $v = (v_1, \dots, v_n)\in \NN^n$ and $m\in \NN$, the
notation $v\leq m$ means that $v_i\leq m$ for every $i = 1, \dots, n$.
For $\alpha = (\alpha_1, \dots, \alpha_n), \beta = (\beta_1, \dots,
\beta_n)\in \NN^n$, we let $x^\alpha = x_1^{\alpha_1} \cdots
x_n^{\alpha_n}$ and $\partial^\beta = \partial_1^{\beta_1} \cdots
\partial_n^{\beta_n}$ in $A_n(R)$. 
The element 
$x^\alpha \partial^\beta\in A_n(R)$ is called a monomial. 
\end{Definition}

\begin{Proposition}\label{Prop11}
The Weyl algebra $A_n(R)$ is a free $R$-module with a basis consisting of the monomials 
$$
M = \{x^\alpha \partial^\beta \mid \alpha, \beta\in \NN^n\}.
$$
\end{Proposition}
\begin{proof}
See \cite[\S2, Lemma 3]{GK}.
\end{proof}

\subsection{Positive characteristic}

If $\kar(R) > 0$, $A_n(R)$ is a finitely generated module over its center.
The following result is a
consequence of Proposition \ref{Prop00} and Proposition \ref{Prop11}.

\begin{Proposition}\label{LemmaWC}
Suppose that 
$\kar(R) = m > 0$ and let 
$$
C = R[x_1^m, \dots, x_n^m, \partial_1^m, \dots, \partial_n^m]\subset A_n(R).
$$  
Then $C$ is a central subalgebra and $A_n(R)$ is a free
module over $C$ of rank $m^{2n}$ with basis
$$
\{x^\alpha \partial^\beta \mid \alpha, \beta\in \NN^n, 0 \leq \alpha, \beta \leq m-1\}.
$$
\end{Proposition}

\begin{Example}
Consider $R = \ZZ/6\ZZ$. Then $3 x^2$ is a central element in $A_1(R)$, but
$3x^2\not\in R[x^6, \partial^6]$.
\end{Example}

The following result is central in this paper. Here we basically follow 
\cite[Theorem 3.1 and Proposition 3.2]{AdjamagboEssen} in the
proof.

\begin{Theorem}\label{Theorem00}
Suppose that $\kar(R) = p$ and let
$$
C = R[x_1^p, \dots, x_n^p, \partial_1^p, \dots, \partial_n^p]\subset A_n(R),
$$  
where $p$ is
a prime number.
\begin{enumerate}[(i)]
\item
The center of $A_n(R)$ is equal to $C$.
\item
If $X_1, \dots, X_n, D_1, \dots, D_n\in A_n(R)$ satisfy the
commutation relations for the Weyl algebra i.e.,
\begin{align}
\begin{split}
[X_i, X_j] &= 0\\
[D_i, D_j] &= 0\\
[D_i, X_j] &= \delta_{ij}
\end{split}
\end{align}
for $i, j = 1, \dots, n$, then 
$$
\{X^\alpha D^\beta \mid 0 \leq \alpha, \beta \leq p-1\}
$$ 
is a basis for $A_n(R)$ as a module over $C$.
\item Let $\varphi:A_n(R)\rightarrow A_n(R)$ be an R-algebra endomorphism. Then
\begin{enumerate}[(a)]
\item
$\varphi(C)\subset C$.
\item
$\varphi$ is injective/surjective if $\varphi|_C$ is injective/surjective. 
\end{enumerate}
\end{enumerate}
\end{Theorem}
\begin{proof}
Proposition \ref{LemmaWC} implies that $C$ is a central subalgebra and
that $A_n(R)$ is a free module over $C$ with basis
$\{x^\alpha \partial^\beta \mid 0 \leq \alpha, \beta \leq p-1\}$.
Suppose that
$$
z = \sum_{0\leq \alpha, \beta \leq p-1} \lambda_{\alpha, \beta} x^\alpha \partial^\beta
$$
is an element of the center of $A_n(R)$
with $\alpha, \beta\in \NN^n$ and $\lambda_{\alpha, \beta}\in C$. 
If $\lambda_{\alpha, \beta}\neq 0$ for some $(\alpha, \beta)\neq (0,0)$, 
then there exists an element $D\in \{x_1, \dots, x_n, \partial_1, \dots, 
\partial_n\}$ with $[D, z]\neq 0$ by Proposition \ref{Prop00}. This proves
$(i)$.

Let $M = \{X^\alpha D^\beta \mid 0 \leq \alpha, \beta \leq p-1\}$. Applying
$\ad(D_i)$ and $\ad(X_i)$ successively, it follows by Proposition \ref{Prop00}
that $M$ is linearly independent over $C$. Suppose that
$R$ is an integral domain and let $K$ denote the field of fractions of $C$.
Then every element in $A_n(R)$ is a $K$-linear combination 
of elements in $M$. However, the formula in 
Proposition \ref{Prop00} applies to show that the coefficients in 
such a linear combination belong
to $C$ proving that $M$ is a generating set over $C$.

In the proof of $(ii)$ for general rings of characteristic $p$, we may
assume that $R$ is noetherian by replacing $R$ with the $\ZZ$-algebra
generated by the coefficients of $X_1, \dots, X_n, D_1, \dots, D_n$ in
the monomial basis from Proposition \ref{Prop11}.
This assumption provides the existence of finitely
many prime ideals $P_1, \dots, P_m\subset R$ such that
\begin{equation}\label{zeroint}
P_1 \cdots P_m = (0).
\end{equation}

Let $N$ denote the
$C$-submodule of $A_n(R)$ generated by $M$. By the
integral domain case, we have
\begin{equation}\label{iter}
A_n(R) = N + P\, A_n(R)
\end{equation}
for $P = P_1, \dots, P_n$.
By iterating \eqref{iter} we get $N = A_n(R)$ using \eqref{zeroint}. This proves
$(ii)$.

For the proof of $(iii)$, let $X_i  = \varphi(x_i)$ and
$D_i = \varphi(\partial_i)$ for $i = 1, \dots, n$. Then
$X_i^p, D_i^p\in C$ by $(ii)$, since $[X_i, D_j^p] = 0$ and
$[D_i, X_j^p] = 0$ for $i, j = 1, \dots, n$. Therefore
$\varphi(C) \subset C$ and $(a)$ is proved. If
$\varphi|_C$ is injective/surjective, then $\varphi$ is 
injective/surjective again using $(ii)$. 
This proves $(b)$.
\end{proof}

\subsection{The Poisson bracket on the center}
\label{antiquantization}

In this section we recall, following Belov-Kanel and Kontsevich  \cite{Kontsevich1}, how the standard Poisson bracket on the center in
prime characteristic is related to the commutator in the Weyl algebra.

\begin{Definition}\label{DefPoisson}
A Poisson algebra $P$ over $R$ is a 
commutative $R$-algebra with an $R$-bilinear skew-symmetric pairing
$\{\cdot,\cdot\}: P\times P\rightarrow P$ called the Poisson bracket, such that $(P, \{\cdot, \cdot\})$
is a Lie algebra over $R$ and $\{a, \cdot\}: P\rightarrow P$ is a derivation
for every $a\in P$ i.e., the \emph{Leibniz rule} $\{a, b c\} = \{a, b\} c + b \{a, c\}$ holds for every $b, c\in A$.
A Poisson ideal $I\subset P$ is an ideal with the property that
$\{a, x\}\in I$ for every $a\in P$ and $x\in I$.
\end{Definition}

\begin{Example}
In general, a
Poisson bracket $\{\cdot, \cdot\}$ on a Poisson $R$-algebra $A$ generated by 
$\xi_1, \dots, \xi_m\in A$ is uniquely determined by $\{\xi_i, \xi_j\}$ for
$i, j = 1, \dots, m$.

The standard Poisson bracket on 
the polynomial ring $P = R[x_1, \dots, x_n, y_1, \dots, y_n]$ 
is uniquely determined
by
\begin{align*}
\{x_i, x_j\} &= 0\\
\{y_i, y_j\} &= 0\\
\{x_i, y_j\} &= \delta_{ij}
\end{align*}
for $i, j = 1, \dots, n$ and given by the formula 
\begin{equation}\label{PoissonBracket}
\{f, g\} = \sum_{i = 1}^n   \left(
\frac{\partial f}{\partial x_i} \frac{\partial g}{\partial y_i} -
 \frac{\partial f}{\partial y_i} \frac{\partial g}{\partial x_i}\right),
\end{equation}
where $f, g\in P$. 
\end{Example}

\begin{Proposition}\label{PropDetOne}
Assume that $R$ is an integral domain and
let $\varphi$ be an $R$-endomorphism of $P = 
R[x_1, \dots, x_n, y_1, \dots, y_n]$ preserving the Poisson bracket
in \eqref{PoissonBracket} i.e., 
$\{f, g\} = \{\varphi(f), \varphi(g)\}$ for every $f, g\in P$. Then
$\det J(\varphi) = \pm 1$, where $J(\varphi)$ is the Jacobian matrix of
$\varphi$ with columns indexed by $x_1, \dots, x_n, y_1, \dots, y_n$ and 
rows by the coordinate functions of $\varphi$.
\end{Proposition}
\begin{proof}
Let $H$ denote the $2n\times 2n$ skew-symmetric matrix
$$
\begin{pmatrix*}[r]
0 & I_n \\
-I_n & 0
\end{pmatrix*}.
$$
The endomorphism $\varphi$ gives rise to
the $2n\times 2n$ skew-symmetric matrix $H(\varphi)$ with entries 
$\{\varphi_i, \varphi_j\}$, where
\begin{align*}
\varphi_i &= \varphi(x_i)\\
\varphi_{i+n} &= \varphi(y_i)
\end{align*}
for $i = 1, \dots, n$.  
The assumption on $\varphi$ implies that $H(\varphi) = H$. Applying the
determinant to the identity
$$
H(\varphi) = J(\varphi) H J(\varphi)^T
$$
therefore shows that 
$\det J(\varphi) = \pm 1$, since $\det H = 1$.
\end{proof}

\newcommand{\tf}{\tilde{f}}
\newcommand{\tg}{\tilde{g}}

Let $p$ be a prime number and suppose that
$R$ has no
$p$-torsion.
Let $\pi$ denote the canonical map
$A_n(R) \rightarrow A_n(R/(p))$ and
$$
C = (R/pR)[x_1^p, \dots, x_n^p, \partial_1^p, \dots, \partial_n^p].
$$ 
For $\tf, \tg\in A_n(R)$, 
$[\tf, \tg]\in p A_n(R)$ if $\pi(\tf)\in C$ or $\pi(\tg)\in C$.
If $\pi(\tf) = \pi(\tf_1)\in C$ and $\pi(\tg) = \pi(\tg_1)\in C$
for $\tf_1, \tg_1\in A_n(R)$, then
$$
[\tf, \tg] - [\tf_1, \tg_1] = [\tf, \tg-\tg_1] - [\tg_1, f-\tf_1]\in p^2 A_n(R).
$$
 Thus for $f, g\in C$,
\begin{equation}\label{poissonBracket}
\{f, g\} := \pi\left(\frac{[\tf, \tg]}{p}\right)\in A_n(R/p R)
\end{equation}
is independent of the choice of $\tf, \tg\in A_n(R)$ with
$\pi(\tf) = f$ and $\pi(\tg) = g$.

\begin{Proposition}
The operation in \eqref{poissonBracket} is the standard Poisson bracket on the
center $C$ of $A_n(R/pR)$ with
\begin{equation}\label{pois}
\{x^p_i, x^p_j\} = \{\partial^p_i, \partial^p_j\} = 0\quad\text{and}\quad
\{ x_i^p, \partial_j^p \} 
= \delta_{ij}
\end{equation}
for 
$i, j=1,\dots,n$. 
\end{Proposition}
\begin{proof}
From properties of the commutator in the Weyl algebra, 
$\{\cdot,\cdot \}:C\times C \rightarrow A_n(R/pR)$ 
is $R$-linear, skew-symmetric and satisfies
the Leibniz rule and the ``$0$'' bracket rules in \eqref{pois}. Proposition \ref{Prop00}$(ii)$ and Wilson's theorem imply that
$\{ x_i^p,\partial_i^p\} = 1$ for $i = 1, \dots, n$. Therefore
$\{C, C\}\subset C$ and $\{\cdot, \cdot\}$ is
the given standard Poisson bracket on $C$.
\end{proof}

\begin{Proposition}\label{PropPoisson}
Let 
$\varphi : A_n(R) \rightarrow A_n(R)$ 
be an $R$-algebra endomorphism and $\m$ a maximal ideal, such that $\kar(R/\m) = p$. Then 
\begin{equation*}
\{ \varphi_\m(f) , \varphi_\m(g) \} = \{ f, g \}
\end{equation*}
for every $f,g \in C_\m$, where $\varphi_\m$ denotes
the induced endomorphism of $A_n(R/\m)$, $C_\m$ the center of $A_n(R/\m)$ and $\{\cdot, \cdot\}$ is the Poisson bracket
coming from the surjection $A_n(R/pR) \rightarrow A_n(R/\m)$.
\end{Proposition}

\begin{proof}
With the notation above $\m C$ is a Poisson ideal in $C$.
Therefore the surjection $\gamma: A_n(R/p R)\rightarrow A_n(R/\m)$
induces the standard Poisson bracket on the center $C_\m$ of
$A_n(R/\m)$ given by $\{f, g\} := \{F, G\}$, where
$\gamma(F) = f$ and $\gamma(G) = g$. 
Taking Theorem \ref{Theorem00}$(iii)(a)$ and \eqref{poissonBracket} into account, the result follows.
\end{proof}

\subsection{Reduction to positive characteristic}

We recall some well known and useful results for reduction to positive characteristic used in this paper. The set of maximal ideals in $R$ is denoted $\Specm(R)$.

\begin{Theorem}\label{TheoremRedModp}\leavevmode
Suppose that $R$ is a finitely generated integral domain over $\ZZ$. Then
\begin{enumerate}[(i)]
\item\label{FFred}
$R/\m$ is a
finite field for every $\m\in \Specm(R)$ and
\item
$$
\bigcap_{\m\in \Specm(R)} \m = (0).
$$
\item \label{locali}
Let $f\in R$. If $\m\in \Specm(R_f)$, then $\m\cap R \in \Specm(R)$ and
$
R/\m\cap R = R_f/\m.
$
\item\label{eqmodp} Let $k$ denote an algebraically closed field containing $R$.
A set of polynomials $f_1, \dots, f_m\in R[T_1, \dots, T_n]$ has a common zero in
$k^n$ if their reductions have a common zero in $(R/\m)^n$ for every $\m\in \Specm(R)$.
\end{enumerate}
\end{Theorem}

\begin{proof} 
The first three results follow from the fact that $R$ is a \emph{Jacobson ring} 
(see \cite[\textsc{Chapter V}, \S3.4]{Bourbaki}). Notice that the identity
$R/\m\cap R = R_f/\m$ in \eqref{locali} is a consequence of $f\not\in \m\cap R$.
Assume in \eqref{eqmodp} that $f_1, \dots, f_m$ do not have a common zero in $k^n$. Then
Hilbert's Nullstellensatz implies that
$$
\lambda_1 f_1 + \cdots + \lambda_m f_m = r
$$
for $\lambda_1, \dots, \lambda_m\in R[T_1, \dots, T_n]$ and $r\in R\setminus \{0\}$. By \eqref{FFred}, 
there exists $\m\in \Spec(R)$, such that $r\not\in \m$. This shows that $f_1, \dots, f_m$
cannot have a common zero in $(R/\m)^n$ contradicting our assumption. 
\end{proof}

\newcommand{\lM}[1]{#1$-$\mathsf{Mod}}
\newcommand{\rM}[1]{\mathsf{Mod}$-$#1}

\section{Flatness}

Let $\lM{A}$ denote the category of left $A$-modules
and $\rM{A}$ the category of right $A$-modules, where $A$ is a ring.

\subsection{Flat ring homomorphisms}

A ring homomorphism $\varphi: S\rightarrow T$ endows $T$ with the left $S$-module structure $s.t = \varphi(s) t$ and the right $S$-module structure
$t.s = t \varphi(s)$, where $s\in S$ and $t\in T$.  We call $\varphi$ \emph{left flat} if 
$
M\mapsto M\otimes_S T
$
is an exact functor from $\rM{S}$ to $\rM{T}$, \emph{right flat} if 
$
M\mapsto T\otimes_S M
$
is an exact functor $\lM{S}\rightarrow \lM{T}$ and \emph{flat} if it is both left and right flat.

\begin{Lemma}
\label{flatcenter}
Suppose that $R$ has prime characteristic and let  
$\varphi$ be an injective endomorphism of $A_n(R)$.
Then $\varphi$ is right/left flat if and only if its restriction
$\varphi|_C$ to the center $C\subset A_n(R)$ is flat.

\end{Lemma}
\begin{proof}
Let $S = \varphi(A_n(R))$ and $C_S = \varphi(C)$. Then the product map
$C\otimes_{C_S} S \rightarrow A_n(R)$ is an isomorphism by
Theorem \ref{Theorem00}\,$(ii)$ (see also \cite[\S3.3, Corollary 2]{Tsuchimoto2}).

Assume that the restriction
$\varphi|_C$ is a flat ring homomorphism of
commutative rings.
For a left $S$-module $M$, the
natural isomorphism
$$
A_n(R)\otimes_S M \cong C\otimes_{C_S} M 
$$
of abelian groups therefore shows that $\varphi$ is right flat. 
The
``opposite'' product map  
$S\otimes_{C_S} C \rightarrow A_n(R)$ similarly shows that $\varphi$ is left flat if $\varphi|_C$ is 
flat.

Suppose that $\varphi$ is left (or right) flat i.e., $A_n(R)$ is flat as a left (or right) module over the subring $S$. Then $A_n(R)$ is a flat $C_S$-module, since $C_S\subset S \subset A_n(R)$ and $S$ is a free $C_S$-module. This implies that $C$ is a flat $C_S$-module as the second step of the extension
$C_S\subset C\subset A_n(R)$ is free and therefore faithfully flat.
\end{proof}

Notice that an injective endomorphism as in Lemma \ref{flatcenter} is right flat if and only if it is left flat.

\subsection{Failure of flatness in positive characteristic}

Let $k$ be a field of positive characteristic $p>0$.
Consider
the endomorphism $\varphi:A_1(k) \rightarrow A_1(k)$ given by
\begin{align*}
\varphi(x) &= x\\
\varphi(\partial) &= \partial + x^{p-1} \partial^p.
\end{align*}

In this section we will prove that $\varphi$ is not a flat ring endomorphism
by showing that the restriction of $\varphi$ to the center of $A_1(k)$
fails to be flat as an endomorphism of commutative rings.

The computation of the restriction to the center
can be quite difficult potentially involving complicated $p$-th powers in
the Weyl algebra.
For the 
endomorphism $\varphi$, a classical formula from Jacobson's book 
\cite[p. 187]{Jacobson} helps greatly in an otherwise complicated
computation: suppose that
$A$ is a ring of prime characteristic $p$ and $a, b\in A$. Then
\begin{equation}\label{JacobsonFormula}
(a + b)^p = a^p + b^p + \sum_{i = 1}^{p-1} s_i(a, b),
\end{equation}
where $i s_i(a, b)$ is the coefficient of $t^{i-1}$ in
$$
D^{p-1}(a),
$$
where $D: A[t]\rightarrow A[t]$ is the derivation $\ad(t a + b)$ and $t$ is a central indeterminate.

\begin{Lemma}
The formula
$$
\left(\partial + x^{p-1} \partial^p
\right)^p = (x^p)^{{}^{p-1}} (\partial^p)^{{}^p}
$$
holds for $x, \partial\in A_1(k)$.
\end{Lemma}
\begin{proof}
This is a consequence of Jacobson's formula \eqref{JacobsonFormula} with
$a = \partial$ and $b = x^{p-1} \partial^p$ using $(p-1)! = -1$ (Wilson's theorem).
\end{proof}

Consider now the restriction of $\varphi$ to the center in terms of
the ring homomomorphism $f:k[u, v]\rightarrow k[u, v]$ given by
\begin{align*}
f(u)&= u\\
f(v)&= u^{p-1} v^p.
\end{align*}
This ring homomorphism is injective and
therefore $\varphi$ is injective by Theorem \ref{Theorem00}.
From \cite[Theorem 7.4]{Matsumura}, we have for a flat ring homomorphism $A\rightarrow B$ that
$$
IB \cap JB = (I \cap J)B
$$
where $I$ and  $J$ are  ideals in $A$. Now let $A=k[u,u^{p-1}v^p]$, $B =k[u,v]$, $I=(u^{p-1})$ and J$=(u^{p-1}v^p)$.

Since $A$ is isomorphic to a polynomial ring in the variables $u$ and  $u^{p-1}v ^p$, 
$$
I \cap J = (u^{2(p-1)} v^p)\subset A.
$$
But $u^{p-1}v^p \in IB \cap JB$ and $u^{p-1} v^p\not\in (u^{2(p-1)} v^p) B$. This proves that $f$
is not a flat ring homomorphism and therefore $\varphi$ is not flat by
Lemma \ref{flatcenter}.

\subsection{Flatness in characteristic zero}\label{FlatCharZero}

In
\cite[Theorem 5.1]{Tsuchimoto3} Tsuchimoto claims a proof of flatness of an
endomorphism of $A_n(K)$ with $K$ is a field of characteristic zero. At a crucial
point he uses that $A_n(R)$ has no ``albert holes'' \cite[Corollary 4.2]{Tsuchimoto3},
where $R$ is a Dedekind domain. However, in the proof of Corollary 4.2, he
applies the wrong statement that $M\mapsto \Gr(M)$ commutes with $-\otimes R/I$ for
a filtered $A$-module $M$ with $R$ commutative and $A$ an almost
commutative $R$-algebra. Also, \cite[Proposition 5.6(4)]{Tsuchimoto3} seems not to be obtained in 
\cite{Tsuchimoto2} or \cite{Tsuchimoto1} as stated in \cite{Tsuchimoto3}. It is not clear to
us how to repair these shortcomings in a straightforward manner. 

In this section we present our proof of flatness of  endomorphisms of Weyl algebras over 
fields of characteristic zero.

\subsubsection{Good filtrations}

%Let $R$ be a commutative ring. 
The \emph{degree}
of a monomial 
$x^\alpha \partial^\beta\in A_n(R)$ 
is defined as 
$\deg(x^\alpha \partial^\beta):=|\alpha| + |\beta|$, where
$|\alpha| = \alpha_1 + \cdots + \alpha_n$ and
$|\beta| = \beta_1 + \cdots + \beta_n$ for
$\alpha = (\alpha_1, \dots, \alpha_n), \beta = (\beta_1, \dots, \beta_n)\in \NN^n$.

The increasing sequence $B = B_0 \subset B_1 \subset \cdots$
of finite rank free $R$-submodules given by
$$
B_m = \Span_R \left\{x^\alpha \partial^\beta \bigm| \deg(x^\alpha \partial^\beta) \leq m\right\} \subset A_n(R)
$$
is a filtration of $A_n(R)$ (called \emph{the Bernstein filtration}) i.e., 
 $\bigcup^\infty_{i=0} B_i = A_n(R)$
and 
$B_i B_j \subset B_{i+j}$ 
for $i, j\in \NN$. Furthermore, 
$$
\Gr_B(A_n(R)) = B_0 \oplus B_1/B_0 \oplus \cdots 
$$ 
is the commutative polynomial ring over $R$ in the $2n$ variables
$[x_1], \dots, [x_n], [\partial_1], \dots, [\partial_n]\in B_1/B_0$ (see
\cite[\S2]{GK}).

Let $M$ be a left module over $A_n(R)$.
A \emph{good filtration} of $M$ is an increasing sequence $0 = M_{-1} \subset M_0 \subset M_1 \subset \cdots$ of
finitely generated $R$-submodules of $M$ with
$\bigcup_i M_i = M$
and 
$B_i M_j \subset M_{i+j}$ 
for $i, j\in \NN$, such that the graded module 
$$
\Gr(M) = M_0 \oplus M_1/M_0 \oplus
\cdots 
$$ 
is finitely generated over $\Gr_B(A_n(R))$ .
If $M$ is finitely generated by $m_1, \dots, m_r\in M$, then
$M_i = B_i m_1 + \cdots + B_i m_r$, $i = 0, 1, 2, \dots$,  is a good filtration of $M$.

\subsubsection{Grothendieck generic freeness}\label{SectGGF}

The Grothendieck generic freeness lemma \cite[Expos\'e IV, Lemme 6.7]{Grothendieck} is a very
important tool 
in our proof of flatness. We need it in the slightly strengthened version presented in 
\cite[Theorem 14.4]{Eisenbud}.

\begin{Lemma}\label{GGF}
Let $A$ be a noetherian integral domain, $B$ a finitely generated $A$-algebra and
$M$ a finitely generated $B$-module. Then there exists $f\in A$, such that
$M_f$ is a free $A_f$-module. If in addition, $B$ is positively graded, with $A$ acting in degree
zero, and if $M$ is a graded $B$-module, then $f$ may be chosen so that each graded component of
$M_f$ is free over $A_f$. 
\end{Lemma}

\begin{Proposition}\label{PropGGF}
Let $R$ denote a noetherian integral domain and $M$ a finitely generated $A_n(R)$-module.
Then there exists a nonzero element $f \in R$, 
such that $M_f$ is a free $R_f$-module.
\end{Proposition}
\begin{proof}
As $M$ is a finitely generated $A_n(R)$-module, it has a good filtration $\{ M_i \}_{i \in \NN}$.
The associated graded module $\Gr(M)$ is a finitely generated module over a polynomial 
ring with coefficients in $R$, and thus, by Lemma \ref{GGF}, there exists 
a nonzero $f \in R$ such that the $R$-modules $(M_i/M_{i-1})_f$ are free,
for $i \in \NN$. By 
choosing a basis for each $(M_i/M_{i-1})_f$, $i \in \NN$, and lifting the collection of these
elements to $M_f$ we obtain a basis of $M_f$. In particular, $M_f$ is free.
\end{proof}

\begin{Lemma}
\label{zeromodulop}
Let $R$ denote a finitely generated integral domain over $\ZZ$, $M$ 
a finitely generated $A_n(R)$-module and let 
$$
N = \bigcap_{\m\in \Specm(R)} \mathfrak{m} M.
$$
Then there exists $f\in R$ with $N_f = 0$.
\end{Lemma}
\begin{proof}
Applying Proposition \ref{PropGGF} to the
finitely generated module $M/N$, we may find
a nonzero $g \in R$, such that $(M/N)_g$ is a free 
$R_g$-module. Similarly we may find a nonzero 
$h \in R$, such that $N_h$ is a free $R_h$-module, since
$M$ is left noetherian as a finitely generated
left module over $A_n(R)$.
Therefore $N_f$ and $(M/N)_f$ are 
free as modules over $R_f$, where $f=gh$. 
Now consider the short exact sequence 
\begin{equation}
\notag
0 \rightarrow N_f \rightarrow
M_f \rightarrow (M/N)_f
\rightarrow 0.
\end{equation}
and fix a maximal ideal $\mathfrak{m}$ in 
$R_f$. As $(M/N)_f$ is free over $R_f$, we obtain an
induced injective map
\begin{equation}
\label{zeromap}
  N_f \otimes_{R_f} R_f/ \mathfrak{m} \rightarrow
M_f \otimes_{R_f} R_f/ \mathfrak{m}.  
\end{equation}
By \eqref{locali} in Theorem \ref{TheoremRedModp}, the
field $R_f/\m$ is isomorphic to $R/ \m_R$,
where $\m_R = \m\cap R$.
In particular, the map (\ref{zeromap}) is identified with the map
\begin{equation}
\notag
N/\m_R N =   N \otimes_{R} R/ \m_R \rightarrow
M \otimes_{R} R/ \m_R = M/ \m_R M,  
\end{equation}
which is zero by the definition of $N$.  
Therefore
$$
N_f \otimes_{R_f} R_f/\mathfrak{m}=0
$$
and $N_f = 0$ by the freeness of $N_f$.
\end{proof}

We now prove flatness in characteristic zero 
by reducing to positive characteristic.

\subsubsection{Flatness}

\begin{Theorem}
Let $K$ be a field of characteristic zero. Then
an endomorphism $\varphi:A_n(K)\rightarrow A_n(K)$
is flat.
\end{Theorem}
\begin{proof}
Let $S = \varphi(A_n(K))$.
We will prove that $A_n(K)$ is flat as a left $S$-module ($\varphi$ is injective, since $A_n(K)$ is a simple ring). The proof that $A_n(K)$ is flat as a right 
$S$-module is similar and is left to the reader.
It suffices to prove that the multiplication
map
\begin{equation}
\label{flatmultmap}
 I \otimes_S A_n(K) \rightarrow A_n(K)
\end{equation} 
is injective for every finitely generated right ideal $I$ in $S$.
Let $M$ denote the right $A_n(K)$-module
$I \otimes_S A_n(K)$, and assume that $m \in M$ maps 
to zero under (\ref{flatmultmap}). We will
prove that $m$ is zero.

Assume that $I$ is generated as a right 
ideal in $S$ by elements 
$\varphi(a_i)$, for $i=1,2,\dots,m$, with 
$a_i \in A_n(K)$. We may then write
$$ m = \sum_{i=1}^m \varphi(a_i) \otimes b_i,$$
for certain elements $b_i \in A_n(K)$. 
Now fix a finitely generated $\ZZ$-subalgebra 
$R$ of $K$, such that all the elements 
$a_1,a_2,\dots,a_m$, $b_1,b_2 \dots ,b_m$
and $\varphi(x_i), \varphi(\partial_i)$,
for $i=1,2,\dots,n$, are contained in 
$A_n(R) \subseteq A_n(K)$. Then there
exists an induced endomorphism
$$
\varphi_R : A_n(R) \rightarrow A_n(R)
$$
whose base change to $K$ equals $\varphi$.
We let $S_R$ denote the image of $\varphi_R$, 
and let $I_R$ denote right ideal in $S_R$
generated by $\varphi(a_1),\varphi(a_2),
\dots, \varphi(a_m) \in S_R$. Finally 
we let $M_R$ denote $I_R \otimes_{S_R} 
A_n(R)$ and let $m_R$ denote the element
$\sum_i \varphi(a_i) \otimes b_i$ in $M_R$.
The base change of $S_R$, $I_R$ and $M_R$
to $K$ then equals $S$, $I$ and $M$ respectively,
and the multiplication map $M_R \rightarrow 
A_n(R)$ will base change to \eqref{flatmultmap}.
Moreover, $m$ equals $m_R \otimes 1$ in $M_R
\otimes_R K=M$.

It suffices to prove that $m_R$ is zero in
some localization $(M_R)_f$, for $f
\in R$. By Lemma \ref{zeromodulop}
this will follow if $m_R$ is zero modulo
every maximal ideal ${\mathfrak m}$ of
$R$. So fix a maximal ideal ${\mathfrak m}$
of $R$, and let $\overline{m}$ denote the
image of $m_R$ in $M_R \otimes_R R/\m$. 
Consider also 
the induced morphism $\overline{\varphi} : A_n(R/\m)
\rightarrow A_n(R/\m)$ with image $S_{R/\m}$. Here $R/\m$ is a field 
of positive characteristic by \eqref{FFred} of Theorem \ref{TheoremRedModp}.

By Proposition \ref{PropPoisson}, the
induced map $\overline{\varphi}|_C : C \rightarrow 
C$ on the center $C$ of $A_n(R/\m)$ preserves the canonical Poisson bracket. 
Therefore $\det(\overline{\varphi}_C) = \pm 1$ by Proposition \ref{PropDetOne} and
$\varphi_C$ is flat by 
\cite[(2.1) \textsc{Theorem}]{Bass}\cite{Miyanishi}.
By Lemma 
\ref{flatcenter}, this shows that $A_n(R/\m)$
is flat as a left module over $S_{R /\m}$. Letting 
$I_{R /\m}$ denote the right 
ideal in $S_{R /\m}$ generated by the image
of $I_R$ in $S_{R /\m}$, it follows that 
the multiplication map 
\begin{equation}
\label{mmodulom} 
M_R \otimes_R R/\m = I_{R /\m} \otimes_{S_{R /\m}} A_n({R /\m}) 
\rightarrow  A_n({R /\m}) ,
\end{equation} 
 is injective. But $\overline{m}$ maps to 
 zero under (\ref{mmodulom}) as 
  $m_R\in M_R$ maps to zero in $A_n(R)$
under the multiplication map   $M_R \rightarrow 
A_n(R)$. We conclude that
 $\overline{m}$ is zero as claimed.
\end{proof}

\section{Automorphisms and polynomial equations}

Let $f\in A_n(R)\setminus\{0\}$. The degree ($\deg f$) of $f$ is defined as
the maximum of the degrees of the monomials occuring with non-zero coefficient 
in the monomial expansion
of $f$ from Proposition \ref{Prop11}. Notice that $\deg(f g) = 
\deg(f) + \deg(g)$ if $f, g\in R\setminus \{0\}$ and $R$ is an integral
domain. 
%If $R$ is an integral domain, then
%$\deg(f g) = \deg(f) + \deg(g)$ for $f, g \in A_n(R)\setminus \{0\}$.

\begin{Definition}
Let $\varphi$ be an endomorphism of $A_n(R)$. Then
the degree of $\varphi$ is defined as 
$$
\deg \varphi = \max\{
\deg \varphi(x_1), \deg \varphi(\partial_1), \dots, 
\deg \varphi(x_n), \deg \varphi(\partial_n)\}.
$$
\end{Definition}

The following result comes from \cite[Proposition 4.2]{Tsuchimoto1}.

\begin{Lemma}\label{LemmaBound}
Let $k$ be a field of prime characteristic $p$ and $\varphi$ an automorphism of $A_n(k)$. Then
$$
\deg(\varphi^{-1}) \leq \deg(\varphi)^{2n-1}.
$$
\end{Lemma}
\begin{proof}
If $\varphi$ is an automorphism, then the induced endomorphism $\varphi|_C$ (see (iii) of Theorem \ref{Theorem00}) of the center $C\subset A_n(k)$ is an automorphism. The bound
on degrees is therefore a
consequence of \cite[(1.4) \textsc{Corollary}]{Bass}, since
$$
\deg \varphi  = \max\{\deg \varphi(x_1)^p, \deg \varphi(\partial_1)^p, \dots, 
\deg \varphi(x_n)^p, \deg \varphi(\partial_n)^p\}/p.
$$
\end{proof}

\begin{Lemma}\label{Lemmaink}
Let $K\subset L$ be a field extension with $\kar(K) = 0$. If $\varphi$
is an automorphism of $A_n(L)$ and
$\varphi(x_i), \varphi(\partial_i)\in A_n(K)$, then
$\varphi^{-1}(x_i), \varphi^{-1}(\partial_i)\in A_n(K)$ for $i = 1, \dots, n$.
\end{Lemma}
\begin{proof} We may write
\begin{equation}\label{bothsides}
x_i = \sum_{\alpha, \beta\in\NN^n} \lambda^i_{\alpha\beta} \varphi(x)^\alpha \varphi(\partial)^\beta\qquad\text{and}\qquad
\partial_i = \sum_{\alpha, \beta\in\NN^n} \mu^i_{\alpha\beta} \varphi(x)^\alpha \varphi(\partial)^\beta
\end{equation}
for $i = 1, \dots, n$ and $\lambda^i_{\alpha\beta}, \mu^i_{\alpha\beta}\in L$.
Applying $\ad(\varphi(x_i))$ and $\ad(\varphi(\partial_j))$ to both sides
of the two identites in \eqref{bothsides} 
gives $\lambda^i_{\alpha\beta}, \mu^i_{\alpha\beta}\in K$ by use of Proposition \ref{Prop00}.
\end{proof}

\begin{Proposition}\label{PropAutoModp}
Let $K$ denote a field with $\kar(K) = 0$ and $R\subset K$ a 
finitely generated $\ZZ$-subalgebra. If 
$\varphi$ is an endomorphism of $A_n(R)$,
such that the induced endomorphism of $A_n(R/\m)$ is an automorphism for every
$\m\in \Specm(R)$, then $\varphi$ is an automorphism of $A_n(K)$.
\end{Proposition}
\begin{proof}
A potential inverse to $\varphi\in A_n(K)$ may be viewed as the
solution to a set of polynomial equations with coefficients in $R$ as
follows. We are looking for elements 
\begin{align}\label{candidates}
\begin{split}
q_i &= \sum_{\alpha, \beta\in \NN^n} \lambda^i_{\alpha\beta} x^\alpha \partial^\beta\qquad\qquad [\text{candidate for }\varphi^{-1}(x_i)]\\
p_i &= \sum_{\alpha, \beta\in \NN^n} \mu^i_{\alpha\beta} x^\alpha \partial^\beta\qquad\qquad [\text{candidate for }\varphi^{-1}(\partial_i)],
\end{split}
\end{align}
for $i = 1, \dots, n$ in $A_n(K)$, such that 
\begin{align}\label{manyeqs}
\begin{split}
[p_i, p_j] &= 0\\
[q_i, q_j] &= 0\\
[p_i, q_j] &= \delta_{ij}\\
x_i &= \sum_{\alpha, \beta\in \NN^n} \lambda^i_{\alpha\beta} \varphi(x)^\alpha \varphi(\partial)^\beta\\
\partial_i &= \sum_{\alpha, \beta\in \NN^n} \mu^i_{\alpha\beta} \varphi(x)^\alpha \varphi(\partial)^\beta
\end{split}
\end{align}
for $i, j = 1, \dots, n$. Using Proposition \ref{Prop11}, the equations in \eqref{manyeqs} may be considered as a system of polynomial equations with coefficients in $R$ in the finitely 
many variables $\lambda^i_{\alpha\beta}, \mu^i_{\alpha\beta}$ for
$i = 1, \dots, n$ and $\alpha, \beta\in \NN^n$ with $|\alpha| + |\beta| \leq \deg(\varphi)^{2n-1}$. By assumption and Lemma \ref{LemmaBound}, this system
of polynomial equations has a solution in $R/\m$ for every $\m\in \Specm(R)$.
By \eqref{eqmodp} in Theorem \ref{TheoremRedModp}, the polynomial system
therefore has a solution $\lambda^i_{\alpha\beta}, \mu^i_{\alpha\beta}\in \overline{K}$ i.e., $\varphi$ is an automorphism in $A_n(\overline{K})$ and
thus an automorphism in 
$A_n(K)$ by
Lemma~\ref{Lemmaink}.
\end{proof}

\section{Birational endomorphisms}\label{SectBirat}

Let $A$ be an Ore domain i.e.,
$$
s A \cap t A\neq (0)\qquad \text{and} \qquad A s \cap A t\neq (0)
$$
for every $s, t\in S = A\setminus\{0\}$. Then $A$ embeds in a division ring $Q(A)$, such that
\begin{enumerate}
\item $Q(A) = \{s^{-1} a \mid a\in A, s\in S\}$.
\item Any homomorphism $f: A\rightarrow T$, such that $f$ maps $S$ 
to invertible elements in $T$  factors through $Q(A)$. \quad
[\emph{universal property}]
\end{enumerate}
The division ring $Q(A)$ is uniquely determined up to isomorphism.

An injective ring homomorphism
$\varphi: A \rightarrow B$ between Ore domains $A$ and $B$ induces a natural injection
$Q(A) \subset Q(B)$.
We call $\varphi$ \emph{birational} if $Q(A) = Q(B)$. 
The Weyl algebra over a noetherian integral domain is an Ore domain, since
it is a left and right noetherian domain.

\begin{Lemma}
\label{centerinkarp}
Let $R$ denote an integral domain of prime characteristic
$p$, $C$ the center of $A_n(R)$ and $K$ the fraction
field of $R$. Then the multiplication map
\begin{equation}
\label{multiso}
 A_n(R) \otimes_C K  \rightarrow Q(A_n(R)),
\end{equation}
is an isomorphism of rings. In particular, the center
of $Q(A_n(R))$ equals $K$. Moreover, an injective 
endomorphism $\varphi : A_n(R) \rightarrow A_n(R)$ is
birational if and only if the induced polynomial map $\varphi|_C : C \rightarrow C$ is birational.
\end{Lemma}
\begin{proof}
By Theorem \ref{Theorem00}, $A_n(R)$ 
is a $C$-algebra and $A_n(R) \otimes_C K$ is a finite dimensional  $K$-algebra 
containing $A_n(R)$ as a subring. Since $A_n(R)$ is a
domain, it follows that $A_n(R)\otimes_C K$ is a division ring \cite[p. A227]{Revoy}.
Thus by the universal property
of $Q(A_n(R))$, the map (\ref{multiso}) has 
an inverse and must be an 
isomorphism. The claim about the center of 
$Q(A_n(R))$ follows, since the center of $A_n(R)$ is $C$.

Consider an injective endomorphism $\varphi : A_n(R) \rightarrow A_n(R)$. 
By (iii) of Theorem \ref{Theorem00} 
we have a commutative diagram 
$$
\xymatrix{
K \ar[d] \ar[r]^\varphi
%K \ar@{^{(}->}[d] \ar[r]^\varphi
  & K \ar[d] \\
%  & K \ar@{^{(}->}[d] \\
Q(A_n(R)) \ar[r]^\varphi & Q(A_n(R))}
$$
of extensions of division rings,  
where the two vertical extensions are of degree 
$p^{2n}$. It follows that the top horizontal extension 
is of degree one if and only if the lower 
horizontal
extension is of degree one. This is equivalent to
the final claim.
\end{proof}

If $K$ is a field of characteristic zero, recall that an endomorphism
of $A_n(K)$ is injective, since $A_n(K)$ is a 
simple ring. It follows that an endomorphism
of $A_n(K)$ induces an endomorphism of $Q(A_n(K))$
which is finite in the following sense.

\begin{Proposition}\label{PropDegBound}
Let $K$ denote a field of characteristic zero and
let $\varphi$ denote an endomorphism of $A=A_n(K)$. 
Let $S=\varphi(A)$ denote the image of $\varphi$.
Then the dimension of $Q(A)$ as a (left or right) module over 
$Q(S)$ is less than or equal to $(\deg \varphi)^{2n}$.  
\end{Proposition}
\begin{proof}
We will prove the bound for the left dimension (the proof for 
the right dimension is similar).
Let $e_1,\dots,e_r   \in Q(A) $ be linearly independent over $Q(S)$. 

By clearing denominators we may assume that  
$e_1,e_2,\dots,e_r$ are elements in $A$. Choose 
$D\in \NN$, such that 
$\deg e_i \leq D$ for $i = 1, \dots, r$. For $j\in \NN$, let
$B_j$ denote the Bernstein
filtration of $A$ and define
$$ 
M_j= \varphi(B_j) e_1 + \varphi(B_j) e_2 +
\dots + \varphi(B_j) e_r.
$$
Then $M_j \subseteq B_{d \cdot j + D}$,   where $d=\deg \varphi$. 
By the linear independence
of $e_1,e_2,\dots,e_r$,
$$
\dim_K(M_j) = r  \dim_K(B_j).
$$ 
This leads to the inequality
$$ r  \dim_K(B_j) \leq \dim_K(B_{d j + D}),$$
for $j\in \NN$. As 
$$\dim_K(B_j)= \frac{1}{(2n)!} j^{2n} + \text{lower degree terms in $j$},$$
we conclude that 
$$
\frac{r}{(2n)!}\leq  \frac{d^{2n}}{(2n)!},
$$ 
which
gives $r \leq d^{2n}$ as claimed.
\end{proof}

\begin{Theorem}
\label{biratisisofield}
Let $K$ denote a 
field of characteristic zero
and let $\varphi$ be an 
endomorphism of $A_n(K)$. If $\varphi$ is birational, then $\varphi$ is an automorphism.
\end{Theorem}
\begin{proof}
The birationality of $\varphi$ implies the 
existence of $a_i,b_i,c_i,d_i \in A_n(K)$,
for $i=1,\dots,n$, such that 
 \begin{align}\label{birat}
\begin{split}
\varphi(a_i) &= x_i \varphi(b_i)  \\
\varphi(c_i) &=  \partial_i \varphi(d_i) \\
\varphi(b_i) & \neq 0 \\
\varphi(d_i) & \neq 0.
\end{split}
\end{align}
Let $T$ denote a finitely generated $\ZZ$-subalgebra of $K$, such that
all the coefficients of $\varphi(x_i), \varphi(\partial_i)$, $a_i$,
$b_i, c_i$ and $d_i$ in the monomial $K$-basis of $A_n(K)$ (see
Proposition \ref{Prop11}), are contained in $T$.  We define $f \in T$
to be the product of all the non-zero coefficients occurring in the
expansions of $b_i$ and $d_i$, for $i=1,2,\dots,n$, in the monomial
$T$-basis of $A_n(T)$. Let $R = T[1/f]\subset K$. 

For $\m\in \Specm(R)$, we let
$C_\m$ denote the center of $A_n(R/\m)$ and $\varphi_\m$ the induced
endomorphism of $A_n(R/\m)$. Notice that $\varphi_\m$ is injective by
(iii) of Theorem \ref{Theorem00}, since $\varphi_\m|_{C_\m}$ is
injective as $\det J(\varphi_\m|_{C_\m}) = \pm 1$ by Proposition
\ref{PropDetOne}.  Since the relations \eqref{birat} are preserved for
$\varphi_\m$, it follows that $\varphi_\m|_{C_\m}$ is birational and
by Lemma \ref{centerinkarp}, that
$\varphi_\m|_{C_\m}$ is a birational endomorphism of $C_\m$. Therefore
$\varphi_\m|_{C_\m}$ is an automorphism by \cite[(2.1)
  \textsc{Theorem}]{Bass} and $\varphi_\m$ is an automorphism by
(iii) of Theorem \ref{Theorem00}.

Now Proposition \ref{PropAutoModp} applies to show that
$\varphi$ is an automorphism of $A_n(K)$.
\end{proof}

\begin{Remark}
We end this paper with two natural questions for
an endomorphism $\varphi$ of
$A_n(K)$, where $K$ is a field of characteristic zero. 
Let $A = A_n(K)$  and
$S = \varphi(A)\subset A$.

\begin{enumerate}[(i)]
\item
Do the left- and right dimensions of $Q(A)$ over $Q(S)$ agree?
\item
Is $S = A$ if $A$ is a finitely generated $S$-module?
\end{enumerate}

The first question could perhaps be answered affirmatively by
reducing to positive characteristic (where the left and right
dimensions do agree).
The last question is inspired by the simply connectedness of
$\CC^n$ 
in the commutative case  \cite[(2.1) \textsc{Theorem}, (e)]{Bass}. 
\end{Remark}

\bibliographystyle{amsplain}
\providecommand{\bysame}{\leavevmode\hbox to3em{\hrulefill}\thinspace}
\providecommand{\MR}{\relax\ifhmode\unskip\space\fi MR }
\providecommand{\MRhref}[2]{%
  \href{http://www.ams.org/mathscinet-getitem?mr=#1}{#2}
}
\providecommand{\href}[2]{#2}

\end{document}